\documentclass[11pt,notitlepage,twoside,a4paper]{amsart}
 \usepackage{amsfonts}

\usepackage{amsmath,amssymb,enumerate}

\usepackage{epsfig,fancyhdr,color}

\usepackage{amssymb}
\usepackage{amsmath,amsthm}
\usepackage{latexsym}
\usepackage{amscd}
\usepackage{psfrag}
\usepackage{graphicx}
\usepackage[latin1]{inputenc}
\usepackage[all]{xy}
\usepackage[mathcal]{eucal}

\definecolor{NoteColor}{rgb}{1,0,0}

\renewcommand{\textsc}{\textcolor{red}}

\newtheorem{theorem}{\rm\bf Theorem}[section]
\newtheorem{proposition}[theorem]{\rm\bf Proposition}
\newtheorem{lemma}[theorem]{\rm\bf Lemma}
\newtheorem{corollary}[theorem]{\rm\bf Corollary}
\newtheorem*{theorem 1}{\rm\bf Proposition 1}
\newtheorem*{theorem 2}{\rm\bf Proposition 2}

\theoremstyle{definition}

\theoremstyle{remark}
\newtheorem{remark}[theorem]{\rm\bf Remark}

\newtheorem{questions}[theorem]{\rm\bf Questions}

\def\interieur#1{\mathord{\mathop{\kern 0pt #1}\limits^\circ}}

\title[The horofunction compactification of Teichm\"uller metric]
{The horofunction compactification of Teichm\"uller metric}

\author{Lixin  Liu}
\address{Lixin Liu, Department of Mathematics, Sun Yat-sen(Zhongshan) University, 510275, Guangzhou, P. R. China}
\email{mcsllx@mail.sysu.edu.cn}

\author{Weixu  Su}
\address{Weixu Su, Department of Mathematics, Fudan University, 200433, Shanghai, P. R. China}
\email{suweixu@gmail.com}

\date{\today}

\thanks{The work is partially supported by NSFC : 10871211}

\begin{document}

\begin{abstract}
We show that the horofunction compactification of Teichm\"uller space, endowed with the Teichm\"uller metric,
is homeomorphic to the Gardiner-Masur compactification.
\end{abstract}

\maketitle


\noindent AMS Mathematics Subject Classification:   32G15 ; 30F60 ; 57M50 :57N05.
\medskip

\noindent Keywords: Teichm\"uller space; Teichm\"uller metric; compactification.
\medskip

\tableofcontents

\section{Introduction}\label{intro}
Let $S=S_{g,n}$ be an oriented surface of genus $g$ with $n$ punctures. We assume that $3g-3+n > 0$.
Denote by $\mathcal{T}(S)$ the Teichm\"uller space of $S$ and $d_T(\cdot, \cdot)$ the Teichm\"uller metric.
It is well-known that the Teichm\"uller metric is a complete Finsler metric, and that any two points in Teichm\"uller space
can be connected by a unique geodesic. For $3g-3+n=1$,
 the Teichm\"uller metric is isometric to the Poincar\'{e} metric on the unit disc, which has constant negative curvature.
 For $3g-3+n>1$, it was proved by Masur \cite{Masur75} that the Teichm\"uller metric
does not have negative curvature in the sense of Busemann, and it was proved by Masur and Wolf \cite{MW} that the Teichm\"uller metric
is not Gromov hyperbolic. For more progress in the study of the Teichm\"uller metric, we refer to Masur \cite {Masur10} and the references there.

There are several compactifications of $\mathcal{T}(S)$, such as Thurston's compactification, Bers' compactification and the Gardiner-Masur  compactification.
Compactification of $\mathcal{T}(S)$ is important in the study of mapping class groups and
degeneration of quasi-Fuchsian groups (see Thurston \cite{Thurston}, Bers \cite{Bers} and Ohshika \cite{Ohshika}).
A natural question is to study the relation between the Teichm\"uller metric  and the compactifications of $\mathcal{T}(S)$.

Teichm\"uller proved that there is a homeomorphism from the open real $(6g-6+2n)$-dimensional ball to $\mathcal{T}(S)$, realized by
the radial map along Teichm\"uller geodesic rays from a fixed base-point. Extending the homeomorphism to the closed ball defines
a compactificaiton of $\mathcal{T}(S)$, which is called Teichm\"uller's compactification and denoted
by $\overline{\mathcal{T}(S)}^T$.

The \emph{ mapping class group}  \index{ mapping class group} is the group of homotopy classes of orientation-preserving homeomorphisms of $S$ and the \emph{extended mapping class group}  is the group of homotopy classes of homeomorphisms of $S$.
Kerckhoff \cite{Kerckhoff} first proved that the action of the mapping class group on $\mathcal{T}(S)$
does not extend continuously to $\overline{\mathcal{T}(S)}^T$. Since the action of the mapping class group on $\mathcal{T}(S)$
extends continuously to Thurston's compactification (Thurston \cite{Thurston}, FLP \cite{FLP}),   it follows that Teichm\"uller's compactification is different from Thurston's compactification.
Masur \cite{Masur} further showed that if a Teichm\"uller geodesic ray is uniquely ergodic
or Strebel, then it converges to a limit point in Thurston's boundary.
There exist Teichm\"uller geodesic rays which do not converge in Thurston's boudary,
see Lenzhen \cite{Lenzhen}.

The aim of this chapter is to relate two distinct compactifications of $\mathcal{T}(S)$:
one is the horofunction compactification (defined by Gromov \cite{Gromov}) with respect to the Teichm\"uller
metric, and such a compactification could be defined on a quite general
class of metric spaces;
the other is  the Gardiner-Masur compactification, defined by using extremal length of simple closed curves.

The main result in this chapter is:
\begin{theorem}\label{Thm:main}
The horofunction compactification
of $(\mathcal{T}(S),d_T)$ is homeomorphic to the Gardiner-Masur compactificaiton of $\mathcal{T}(S)$.
\end{theorem}
The proof of Theorem  \ref{Thm:main} and an explict homeomorphism from the Gardiner-Masur compactificaiton to the horofunction
compactification will be given in Section \ref{main}.

There are two corollaries (see Proposition \ref{pro:mod} and Proposition \ref{pro:geo} for the proofs):

\begin{corollary}\label{corollary1}
The action of the extended mapping class group on $\mathcal{T}(S)$ extends continuously to the
Gardiner-Masur boundary.
\end{corollary}
\begin{corollary}\label{corollary2}
Every Teichm\"uller (almost-)geodesic ray converges in the forward direction to a point in the Gardiner-Masur boundary.

\end{corollary}
 Note that
Miyachi \cite{Miyachi} already gave a different proof of the fact that the action of the mapping class group on $\mathcal{T}(S)$ extends continuously to the
Gardiner-Masur boundary.

Miyachi \cite{Miyachi}, \cite{Miyachi2} also proved that if a Teichm\"uller geodesic ray is uniquely ergodic
or Strebel, then it converges to a limit point in the Gardiner-Masur boundary.
Recall that a Teichm\"uller geodesic ray  is uniquely ergodic or Strebel if the vertical measured foliation of the quadratic differential defining the geodesic ray is uniquely ergodic or Strebel (all of the vertical trajectoris are closed).
Miyachi \cite{Miyachi}, \cite{Miyachi2}  proved that if the Teichm\"uller geodesic ray is uniquely ergodic, the limit point in the Gardiner-Masur boundary is equal to the projective class of the vertical measured foliation (see Theorem \ref{thm:Miyachi} for the statement); and if
 the Teichm\"uller geodesic ray is Strebel, the limit point is determined by the vertical measured foliation and the conformal structure of the initial point of the geodesic ray.

It follows that the Gardiner-Masur compactification is natural and compatible with the Teichm\"uller metric,
although we do not know much of its
 geometric structure.

\bigskip
Boundary points of the horofunction compactification are called horofunctions. A horofunction is called
a Busemann point if it is a limit point of some almost-geodesic ray (see Section 3 for the definition of almost-geodesic ray).
Note that for a Hadamard manifold, i.e., a nonpositively curved and simply connected, the horofunction compatification is the same as the geodesic compactification (see \cite{BGP}). In particular, each
horofunction is a Busemann point.

Since the set of horofunctions of $(\mathcal{T}(S),d_T)$ is identified with the Gardiner-Masur boundary
(by Theorem \ref{Thm:main}), it is natural to ask
whether every point in the Gardiner-Masur boundary is an accumulation point of a Teichm\"uller almost-geodesic ray.

We already pointed out that when $3g-3+n=1$, the Teichm\"uller metric is isometric to the hyperbolic plane. In this case,
each horofunction is a Busemann point.
When $3g-3+n\geq 2$, non-Busemann points of the Gardiner-Masur boundary were recently constructed  by Miyachi \cite{Miyachi3}.
He proved that the
projective class of a maximal rational measured foliation cannot be the limit of any
almost geodesic ray in the Gardiner-Masur compactification.

Inspired by the results of Miyachi, one may asks the following questions:
\begin{questions}
\begin{enumerate}
\item Give a necessary and sufficient conditions for a point in the  Gardiner-Masur boundary to be a Busemann point.  Determine the limit point of a general Teichm\"uller geodesic ray in the Gardiner-Masur boundary from the conformal structure of its initial point and (the vertical measured foliation of) the quadratic differential determining the geodesic ray.
\item  Is the set of Busemann points dense in the Gardiner-Masur boundary? Furthermore, understand the geometric structure of
the Gardiner-Masur boundary.
\end{enumerate}
\end{questions}

\begin{remark}
Recently we learnt from Cormac Walsh that he has solved the first question.
\end{remark}
Our proof of Theorem \ref{Thm:main} is inspired by a recent result of Walsh \cite{Walsh}, which appears in this volume.  He proved  that
Thuston's compactification of $\mathcal{T}(S)$ is homeomorphic to the horofunction
compactification of $\mathcal{T}(S)$ endowed with Thurston's Lipschitz asymmetric metric.
We have felt for some years that the Teichm\"uller metric is natural for the Gardiner-Masur compactification in some sense, while Thurston's Lipschitz asymmetric metric is natural for Thurston's compactification. Now we know that the horofunction compactification  builds a bridge between them.

\textbf{Acknowledgements.}  We would like to thank
Athanase Papadopoulos for his useful
comments and corrections. We are also grateful to Hideki Miyachi for reading carefully  and for his help in the proof of Proposition \ref{pro:inj}.

\section{Preliminaries}

The Teichm\"uller space \index{Teichm\"uller space}
$\mathcal{T}(S)$ of $S$ is the space of complex structures (or complete, finite-area hyperbolic structures) $X$ on $S$
up to equivalence.
We say that two complex structures $X$ and $Y$ are equivalent if there is a conformal map $h:X \to Y$
homotopic to the identity map on $S$.

The Teichm\"uller metric \index{Teichm\"uller metric}   on $\mathcal{T}(S)$ is the metric defined by
$$d_T(X,Y):=\frac{1}{2} \inf_f \log K(f)$$
where $f:X \to Y$ is a
quasi-conformal map homotopic to the identity map of $S$ and

$$\underset{x\in X}{\mathrm{ess}-\sup} \  K_x(f)\geq 1$$

 is the quasi-conformal dilatation of $f$, where
 $$K_x(f)=\frac{|f_z(x)|+|f_{\bar z}(x)|}{|f_z(x)|-|f_{\bar z}(x)|}$$
 is the pointwise quasiconformal dilatation at the point $x\in X$ with local conformal
coordinate $z$.

Teichm\"uller's theorem states that given any $X,Y \in \mathcal{T}(S)$, there exists a unique
quasi-conformal map $f: X \to Y$, called the Teichm\"uller map, such that
$$d_T(X,Y)=\frac{1}{2} \log K(f).$$
The Beltrami differential $\mu:= \frac{\bar\partial f}{\partial f}$ of the Teichm\"uller map $f$ is of the form $\mu=k\frac{\bar q}{|q|}$
for some quadratic differential $q$ on $X$ and some constant $k$ with $0\leq k <1$.
In natural coordinates given by $q$ on $X$ and a quadratic differential $q'$ on $Y$, the Teichm\"uller map
$f$ is given by $f(x+iy)=K^{1/2}x+ i K^{-1/2}y$, where $K=K(f)=\frac{1+k}{1-k}$.

The Teichm\"uller metric is induced by a Finsler norm. Between any two points in $\mathcal{T}(S)$ there is precisely one geodesic.
A geodesic ray with initial point $X$ is given by the one-parameter family of Riemann surfaces $\{X_t\}_{t\geq 0}$,
where there is a holomorphic quadratic differential $q$ on $X$ and a $t$-family of
Teichm\"uller maps $f_t:X\to X_t$, with initial Beltrami differentials $\mu(f_t)=\frac{e^{2t}-1}{e^{2t}+1}\frac{\bar q}{|q|}$. Here  $\mu(f_t)$ is chosen such that the geodesic ray has arc-length parameter given by $t$.

\bigskip
Extremal length \index{extremal length}  is an important tool in the study of the Teichm\"uller metric. The notion is due to Ahlfors and Beurling (ref. \cite{AB}).
Recall that an \emph{essential}  simple closed curve on $S$ is a simple closed curve on $S$
which is neither homotopic to a point on $S$ nor homotopic to a puncture of $S$.
Let $\mathcal{S}$ be the set of homotopy classes of  essential simple closed curves on $S$.

Given a Riemann surface $X$, a conformal metric $\sigma$ on $X$ is a metric locally of the form
$\sigma(z)|dz|$ wher  $\sigma(z)\geq 0$ is a Borel measurable function.  We define the
$\sigma$-area  of $X$ by
$$ A(\sigma)=\int_X \sigma^2(z)|dz|^2 .$$
If $\alpha\in \mathcal{S}$, then the $\sigma$-length of $\alpha$ is defined by
$$L_\sigma(\alpha)=\inf_{\alpha' }\int_{\alpha'}\sigma(z)|dz|,$$
where the infimum is taken over all essential simple closed curves $\alpha' $ in the homotopy class of  $\alpha$.

With the above notation, we can define the extremal length of $\alpha$ on $X$  by
$$\mathrm{Ext}_X(\alpha)=\sup_\sigma \frac{L_\sigma^2(\alpha)}{A(\sigma)},$$
where  $\sigma(z)|dz|$ ranges over all conformal metrics on $X$  with
$0<A(\sigma)<\infty$.

The definition of extremal length only depends on the homotopy class of $X$ and
the homotopy class of $\alpha$. Fix $\alpha\in \mathcal{S}$, $\mathrm{Ext}_X(\alpha)$ defines a function on  Teichm\"uller space.
The following important formula is due to Kerckhoff \cite{Kerckhoff}.
\begin{theorem} Let $X, Y$ be any two points of $\mathcal {T}(S)$. Then
$$d_T(X,Y)=\frac{1}{2} \log \sup_{\alpha\in \mathcal{S}} {\frac{\mathrm{Ext}_X(\alpha)}{\mathrm{Ext}_Y(\alpha)}}.$$
\end{theorem}

A \emph{measured foliation} \index{measured foliation}  on $S$ is a foliation (with a finite
number of singularities) with an invariant transverse measure. The
singularities which are allowed are  topologically  the same as
those that occur at $z=0$ in the line field defined by the quadratic form $z^{p-2}dz^2$. Two
measured foliations $\mu$ and $\mu'$ are
\emph{equivalent} if for any simple closed curve $\gamma$, the geometric
intersection number $i(\gamma,\mu)$ and $i(\gamma,\mu')$ equal. Denote by $\mathcal{MF}$
the space of equivalent classes of measured foliations.

There is a special class of measured foliations that
have the property that the complement of the critical leaves is
homeomorphic to a cylinder. The leaves of the foliation on the
cylinder are then all freely homotopic to a simple closed curve
$\gamma$. Such a foliation is completely determined as a point in
$\mathcal{MF}$ by the height $r$ of the cylinder
 and the isotopy class of $\gamma$. Denote such a foliation by $(\gamma, r)$ and call it a \emph{weighted simple closed curve}. Thurston \cite{Thurston}
 showed that $\mathcal{MF}$ is homeomorphic to a $(6g-6)$
dimensional ball and that there is an embedding $\mathcal{S}\times \mathrm{R}_+\rightarrow \mathcal{MF}$
whose image is dense in $\mathcal{MF}$.

By Kerckhoff \cite{Kerckhoff}, there is a unique
continuous extension of the extremal length function from $\mathcal{S}$ to $\mathcal{MF}$,
with $\mathrm{Ext}_X((\gamma, r))=r^2\mathrm{Ext}_X(\gamma)$. As a result,
the density of (weighted) simple closed curves in $\mathcal{MF}$  allows us to replace the right
hand side of Kerckhoff's formula by the supremum taken over all $\mu\in \mathcal{MF}$:
$$d_T(X,Y)=\frac{1}{2} \log \sup_{\mu\in \mathcal{MF}} {\frac{\mathrm{Ext}_X(\mu)}{\mathrm{Ext}_Y(\mu)}}.$$

Denote the space of projective measured foliations by $\mathcal{PMF}$.
We may identify $\mathcal{PMF}$ with the cross-section $\{\mu\in \mathcal{MF}\ | \ \mathrm{Ext}_{X_0}(\mu)=1\}$ (for any fixed point $X_0\in \mathcal{T}(S)$) and
write Kerckhoff's formula as
$$d_T(X,Y)=\frac{1}{2}\log \sup_{\mu\in \mathcal{PMF}} \frac{\mathrm{Ext}_Y(\mu)}{\mathrm{Ext}_X(\mu)}.$$
Since $\mathcal{PMF}$ is compact, the supremum is attained by some $\mu\in \mathcal{PMF}$.

\bigskip

A measured foliation $\mu$ is \emph{minimal}  if no closed curve in $S$ can be realized by leaves of $\mu$.
Equivalently, after Whitehead moves, the foliation has only dense leaves on $S$.
Two measured foliations $\mu$ and $\nu$ are  \emph{ topologically equivalent} if after Whitehead moves, the leaf structures are isotopic to each other.
A measured foliation $\mu$ is called \emph{ uniquely ergodic} \index{uniquely ergodic}  if  it is minimal and any measured foliation topologically
equivalent to  $\mu$ is measure equivalent to a positive multiple of $\mu$. The following lemma is proved in Masur \cite{Masur}.
\begin{lemma}\label{Lem:Masur}
Assume that $\mu \in \mathcal{MF}$ is  uniquely ergodic. If $ \nu\in \mathcal{MF}$ satisfies $i(\mu,\nu)=0$, then $\nu=c\mu$ for some constant $c\geq 0$.
\end{lemma}
Moreover, it follows from Thurston's theory that
uniquely ergodic measured foliations are dense in $\mathcal{MF}$. This follows from the fact that
the orbit of any element of $\mathcal{PMF}$ by the  mapping class group action is dense in this space.

\section{Compactifications of Teichm\"uller space}
Thurston introduced a compactification $\overline{\mathcal{T}(S)}^{Th}$ of $\mathcal{T}(S)$
such that the action of the mapping
class group on $\mathcal{T}(S)$
extends continuously to the boundary $\partial \overline{\mathcal{T}(S)}^{Th}$ of $\overline{\mathcal{T}(S)}^{Th}$. We recall
some of the fundamental results of Thurston as described in \cite{FLP}. Again
denote by $\mathcal{S}$ the homotopy classes of essential simple closed curves with the discrete
topology. Let $\mathbb{R}_+^{\mathcal{S}}$ be the set of nonnegative functions on $\mathcal{S}$ and let $P(\mathbb{R}_+^{\mathcal{S}})$ be the projective space of
$\mathbb{R}_+^{\mathcal{S}}$. Denote by $\pi: \mathbb{R}_+^{\mathcal{S}} \to P(\mathbb{R}_+^{\mathcal{S}})$
the natural projection. We endow $\mathbb{R}_+^{\mathcal{S}}$ with the product topology and $P(\mathbb{R}_+^{\mathcal{S}})$ with the quotient  topology.
There is a mapping $\tilde{\psi}$ from $\mathcal{T}(S)$ into $\mathbb{R}_+^{\mathcal{S}}$ which sends
$X$ to the function $\tilde{\psi}(X)$ defined by
$$\tilde{\psi}(X)(\alpha)= \mathrm{\ell}_X(\alpha)$$
for all $\alpha\in \mathcal{S}$, where $\mathrm{\ell}_X(\alpha)$ is the hyperbolic length of $\alpha$ on $X$.
Thurston showed that
$\psi=\pi \circ \tilde{\psi}: \mathcal{T}(S) \to P(\mathbb{R}_+^{\mathcal{S}})$ is an embedding and we call it Thurston's embedding.

There is also an embedding of $\mathcal{PMF}$  into $P(\mathbb{R}_+^{\mathcal{S}})$ which sends each projective class of
measured foliation $[\mu]$ to the projective class of the function
$$\gamma \to i(\mu,\gamma),$$
where $i(\mu, \gamma)$ is the geometric intersection number of the measured foliation $\mu$ with the homotopy class of simple closed curve $\gamma$.
Thurston showed that with these embeddings
the closure $\overline{\psi(\mathcal{T}(S))}$ of the image $\psi(\mathcal{T}(S))$ in $P(\mathbb{R}_+^{\mathcal{S}})$
 is homeomorphic to a real $(6g-6+2n)$-dimensional closed ball and $\mathcal{PMF}$ is  the boundary sphere of $\overline{\psi(\mathcal{T}(S))}$.  We let $\overline{\mathcal{T}(S)}^{Th}=\overline{\psi(\mathcal{T}(S))}$ and call it \emph{Thurston's compactification} of $\mathcal{T}(S)$. The complement $\partial \overline{\mathcal{T}(S)}^{Th}=\overline{\psi(\mathcal{T}(S))}-\psi(\mathcal{T}(S))$ is called  \emph{Thurston's boundary} \index{Thurston'scompactification}  of $\mathcal{T}(S)$.
We always identify $\partial \overline{\mathcal{T}(S)}^{Th}$ with $\mathcal{PMF}$ without referring to the embedding.

\bigskip
By replacing the hyperbolic length functions $\ell_X(\alpha)$ by the square root of extremal length functions, Gardiner and Masur \cite{GM} defined the  \emph{Gardiner-Masur compactification}  of
$\mathcal{T}(S)$ and the corresponding boundary  is called the \emph{Gardiner-Masur boundary}.

Now we give more details. Define a mapping $\tilde{\phi}$ from $\mathcal{T}(S)$ into $\mathbb{R}_+^{\mathcal{S}}$ by
$$\tilde{\phi}(X)(\alpha)= \mathrm{Ext}_X(\alpha)^{1/2}$$
for all $\alpha\in \mathcal{S}$. Let $P(\mathbb{R}_+^{\mathcal{S}})$ be as before the projective space of
$\mathbb{R}_+^{\mathcal{S}}$ and $\pi: \mathbb{R}_+^{\mathcal{S}} \to P(\mathbb{R}_+^{\mathcal{S}})$
be the natural projection. Gardiner and Masur \cite{GM} showed that
$\phi=\pi \circ \tilde{\phi}: \mathcal{T}(S) \to P(\mathbb{R}_+^{\mathcal{S}})$ is an embedding
and the closure $\overline{\phi(\mathcal{T}(S))}$ of the image $\phi(\mathcal{T}(S))$ in $P(\mathbb{R}_+^{\mathcal{S}})$ is compact.
Let $\overline{\mathcal{T}(S)}^{GM}=\overline{\phi(\mathcal{T}(S))}$. Then $\overline{\mathcal{T}(S)}^{GM}$ is called the \emph{Gardiner-Masur compactification}   of
$\mathcal{T}(S)$. The complement $\partial \overline{\mathcal{T}(S)}^{GM}=\overline{\phi(\mathcal{T}(S))}-\phi(\mathcal{T}(S))$ is called the
\emph{Gardiner-Masur boundary}   \index{Gardiner-Masur compactification}  of $\mathcal{T}(S)$.

Gardiner and Masur \cite{GM} also proved that $\partial \overline{\mathcal{T}(S)}^{Th}$ is a proper
subset of  $\partial \overline{\mathcal{T}(S)}^{GM}$.
For further investigations about  Thurston's boundary, the Gardiner-Masur boudary and their relations with
the Teichm\"uller geometry, we refer to Gardiner and Masur \cite{GM} and to recent works of Miyachi \cite{Miyachi}, \cite{Miyachi2}. We just recall the following  theorem that will be used later.

\begin{theorem}[Miyachi \cite{Miyachi2}]\label{thm:Miyachi}
Suppose that $r(t)$ is a Teichm\"uller geodesic ray defined by a quadratic differential $q$. If the vertical measured foliation $\mathcal{F}_v(q)$ of $q$ is uniquely ergodic, then $r(t)$ converges to a point on the Gardiner-Masur boundary and the limit is  equal to the projective class of $\mathcal{F}_v(q)$.
\end{theorem}
Miyachi's theorem is an analogue of a theorem of Masur \cite{Masur} saying that if $\mathcal{F}_v(q)$  is uniquely ergodic, then  $r(t)$ converges to a point on Thurston's boundary and that the limit point is  equal to the projective class of $\mathcal{F}_v(q)$.

\bigskip

In the following, we will give the definition of horofunction compactification for a general metric space and
then use Theorem \ref{Thm:main} to explain Corollary \ref{corollary1} and Corollary \ref{corollary2}. The proof
of Theorem \ref{Thm:main} will be  postponed until Section \ref{main}.

Recall that for a locally compact geodesic metric space $(M,d)$, the horofunction compactification is defined
by Gromov \cite{Gromov} in the following way.
Fix a base point $x_0\in M$. To each $z\in M$ we assigned a function $\Psi_z: M \to \mathbb{R}$, defined by
\begin{equation}\label{Equation:Psi}
\Psi_z(x)=d(x,z)-d(x_0,z).
\end{equation}
Let $C(M)$ be the space of continuous functions on $M$ endowed with the  topology
of locally uniformly convergence on $M$.
Then the map $\Psi:M\to C(M), z \mapsto \Psi_z$ is an embedding from $M$ into $C(M)$.
The closure $\overline{\Psi(M)}$ of $\Psi(M)\subset C(M)$ is compact, and it is called the \emph{horofunction compatification}  \index{Horofunction compactification}  of $(M,d)$. The \emph{horofunction boundary} is defined to be
$$M(\infty)=\overline{\Psi(M)}-\Psi(M),$$
and its elements are called \emph{horofunctions}.

Note that here the definition of $M(\infty)$ depends on the choice of the
base point $x_0$. If one changes to an alternative base point $x_1$,
then the assignment of the new function $\Psi'_z$ is
related to $\Psi_z$ by $\Psi'_z(\cdot)=\Psi_z(\cdot)-\Psi_z(x_1)$. One can check that there is a natural
identification between $\overline{\Psi(M)}$ and $\overline{\Psi'(M)}$ and $M(\infty)$ is well-defined up to additive constants.

Let $\mathrm{Isom} (M, d)$ be the isometry group of the metric space $(M, d)$.

\begin{proposition}\label{pro:mod}
The action of the isometry group $\mathrm{Isom} (M, d)$  of $M$
extends continuously to a homeomorphism on the horofunction
compactification.
\end{proposition}
\begin{proof}
For any isometry $g \in \mathrm{Isom} (M, d)$ and
any horofunction $\xi \in M(\infty)$, we define $g \cdot \xi\in C(M)$ to be
$$ (g \cdot \xi)(x)=\xi (g^{-1}\cdot x)-\xi (g^{-1} \cdot x_0).$$
To see that $g \cdot \xi$ is well-defined, assume that $x_n\in M$ converges to $\xi$, then
\begin{eqnarray*}
\lim_{n\to\infty}\Psi_{g\cdot x_n}(x) &=& \lim_{n\to\infty} d(x,g\cdot x_n)-d(x_0,g\cdot x_n)  \\
&=& \lim_{n\to\infty} d(g^{-1}\cdot x, x_n)-d(g^{-1}\cdot x_0, x_n) \\
&=& \xi (g^{-1}\cdot x)-\xi (g^{-1} \cdot x_0).
\end{eqnarray*}
It is easy to check that $g$ induces a homeomorphism from the horofunction
compatification with base-point $x_0$ onto the horofunction
compatification with base-point $g^{-1}\cdot x_0$.

\end{proof}

Let $(\mathcal{T}(S),d_T)$ be the Teichm\"uller space endowed with the Teichm\"uller metric.
Suppose that there is a homeomorphism $\Psi$ between the Gardiner-Masur compactification $\overline{\mathcal{T}(S)}^{GM}$ and
the horofunction compactification $\overline{\Psi(\mathcal{T}(S))}$ (This is the statement of Theorem \ref{Thm:main}).

By Proposition \ref{pro:mod}, any element $g$ of the isometry group $\mathrm{Isom} (\mathcal{T}(S),d_T)$ extends
continuously to a self-homeomorphism of $\overline{\Psi(\mathcal{T}(S))}$. It is easy to see that
the action of $g$ on $\mathcal{T}(S)$ is the same as the action of $\Psi^{-1}\circ g\circ \Psi$. As a result,
we have:
\begin{corollary}\label{corollary3}
The action of $\mathrm{Isom} (\mathcal{T}(S),d_T)$ on $\mathcal{T}(S)$ extends continuously to a homeomorphism on the Gardiner-Masur compactification.
\end{corollary}

Note that the extended mapping class group acts isometrically on the Teichm\"uller space with
the Teichm\"uller metric. In fact, a famous theorem of Royden (generalized by Earle-Kra  \cite{EK}) shows that, if $S$ is neither a sphere
with $\leq 4$ punctures or a torus with $\leq 2$ punctures, then the extended mapping class group
is precisely the group of isometries  of $\mathcal{T}(S)$
with the Teichm\"uller metric (modulo its center if $S$ is closed of
genus $2$). Corollay \ref{corollary1} follows from Corollary \ref{corollary3}.

\bigskip
Recall that a \emph{geodesic ray} in a metric space $(M,d)$ is an embedding $\gamma$ from the interval  $[0,\infty)$  to $M$ such that
$$d(\gamma(s), \gamma(t))=t-s,$$
for all $s,t \in [0,\infty)$, with $s<t$.

A map $\gamma : T \to M$, where $T$ is an unbounded subset of $\mathbb{R}_+$ containing $0$, is called an \emph{almost-geodesic ray} if for any $\epsilon >0$ there exists a $m \ge 0$ such that
$$| d(\gamma(0), \gamma(s)) + d(\gamma(s), \gamma (t)) -t| <\epsilon $$
for all $s, t \in T$ with $m\le s \le t$.
The definition of almost-geodesic ray is due to Rieffel \cite{Rieffel}.

Rieffel \cite{Rieffel} proved the following result:
\begin{proposition}\label{pro:geo}
Let $(M,d)$ be a metric space.
Every almost-geodesic ray of $(M,d)$ converges to a point in $M(\infty)$.
\end{proposition}

Now Corollary \ref{corollary2}, the convergence of  Teichm\"uller almost-geodesic rays to the Gardiner-Masur boundary, follows from Rieffel's result and Theorem \ref{Thm:main}.
We learned after finishing this manuscript that Miyachi \cite{Miyachi3} gave a new  proof of Corollary \ref{corollary2}. Since his proof is short and intrinsic, we state it here for convenience of the readers.
\begin{proof}[Proof of Corollary \ref{corollary2}]
Let $\gamma : T \to \mathcal{T}(S)$ be a Teichm\"uller almost-geodesic ray.  For any $\epsilon >0$, there exists an $m \ge 0$ such that
$$| d_T(\gamma(0), \gamma(s)) + d_T(\gamma(s), \gamma (t)) -t| <\epsilon $$
for all $s, t \in T$ with $m\le s \le t$. In particular, taking $s=t \geq m$ in $T$, we have
$$| d_T(\gamma(0), \gamma(t)) -t| <\epsilon.$$
Let  $\log K_{r(t)}= 2d_T(\gamma(0), \gamma(t))$; then the above inequality is equivalent to

\begin{equation}\label{equ:almost}
 e^{t-\epsilon} \leq K_{r(t)}^{1/2}\leq e^{t+\epsilon} .
\end{equation}

For any $ m\leq s<t$ in $T$, by Kerckhoff's distance formula,
$$\sup_{\alpha\in \mathcal{S}} \frac{Ext_{r(t)}^{1/2}(\mu)}{Ext_{r(s)}^{1/2}(\mu)}=e^{d_T(r(s),r(t))}=e^{t-s}.$$
As a result, for any $\alpha\in \mathcal{S}$,
$$\frac{Ext_{r(t)}^{1/2}(\alpha)}{ K_{r(t)}^{1/2}}/ \frac{Ext_{r(s)}^{1/2}(\alpha)}{ K_{r(s)}^{1/2}}=\frac{Ext_{r(t)}^{1/2}(\mu)}{Ext_{r(s)}^{1/2}(\mu)} \cdot \frac{K_{r(s)}^{1/2}}{K_{r(t)}^{1/2}}\leq e^{t-s} \frac{K_{r(s)}^{1/2}}{K_{r(t)}^{1/2}}.$$
By $ (\ref{equ:almost})$, we have
$$\frac{Ext_{r(t)}^{1/2}(\alpha)}{ K_{r(t)}^{1/2}} \leq e^{2\epsilon} \frac{Ext_{r(s)}^{1/2}(\alpha)}{ K_{r(s)}^{1/2}}.$$

Let $\mathcal{E}_{r(t)}(\alpha)=\frac{\mathrm{Ext}_{r(t)}(\alpha)^{1/2}}{K_{r(t)}^{1/2}}$; then
\begin{equation}\label{equ:dis}
\mathcal{E}_{r(t)}(\alpha) \leq e^{2\epsilon} \mathcal{E}_{r(s)}(\alpha).
\end{equation}

It follows that $\mathcal{E}_{r(t)}(\alpha)$ is a bounded function of $t$. We set
$$\mathcal{E}(\alpha)= \lim\inf_{t \in T} \mathcal{E}_{r(t)}(\alpha).$$

Consider the supremum limit $ \lim\sup_{t \in T} \mathcal{E}_{r(t)}(\alpha)< \infty$.
By $(\ref{equ:dis})$, we have
$$\lim\sup_{t \in T} \mathcal{E}_{r(t)}(\alpha) \leq e^{2\epsilon} \lim\inf_{t \in T} \mathcal{E}_{r(t)}(\alpha).$$
Since $\epsilon$ can be chosen arbitrary small, we have
$$\lim\sup_{t \in T} \mathcal{E}_{r(t)}(\alpha) =\lim\inf_{t \in T} \mathcal{E}_{r(t)}(\alpha).$$
As a result, $\mathcal{E}(\alpha)$ is actually  the limit of $\mathcal{E}_{r(t)}(\alpha)$.

Note that the above argument is uniformly in $\alpha\in \mathcal{S}$, by the definition of the Gardiner-Masur compactification, $\gamma(t)$ converges to a limit in the Gardiner-Masur boundary.
\end{proof}
The  function $\mathcal{E}_{r(t)}(\alpha)$ is important for understanding the geometry of the Gardiner-Masur compactification. Note that if $r(t)$ is a Teichm\"uller geodesic ray,  then
$\mathcal{E}_{r(t)}(\alpha)$ is a discreasing function of $t$. We will see in next section that this function is generalized by Miyachi \cite{Miyachi}  to give a representation of each point in the Gardiner-Masur boundary.

\section{Convergence in the Gardiner-Masur compacfitication}\label{sec:3}

In the following, when we say that
a sequence $P_n\in \overline{\mathcal{T}(S)}^{GM}$ converges to $P\in \overline{\mathcal{T}(S)}^{GM}$,
we always refer to the convergence in the sense of the Gardiner-Masur compactification.

Fix a point $X_0 \in \mathcal{T}(S)$ as the base-point of the horofunction compactification.
For any $X\in \mathcal{T}(S)$, denote by $K_X$  the dilatation of
the Teichm\"uller map between $X_0$ and $X$. Note that $d_T(X_0,X)=\frac{1}{2}\log K_X$
and $K_X=\sup_{\mu\in \mathcal{MF}}\frac{\mathrm{Ext}_X(\mu)^{1/2}}{\mathrm{Ext}_{X_0}(\mu)^{1/2}}$ (here we set the ratio to be $1$ when $\mu=0$).
Consider the following function defined on $\mathcal{MF}$:
\begin{equation}\label{Equation:ext}
\mathcal{E}_X(\mu)=\frac{\mathrm{Ext}_X(\mu)^{1/2}}{K_X^{1/2}}, \mu \in\mathcal{MF}.
\end{equation}

The functions defined in $(\ref{Equation:ext})$ are due to Miyachi \cite{Miyachi} and they
can be continuously extended to the Gardiner-Masur boundary.
They play a role analogous to the intersection numbers $i(\mu, \cdot)$ in Thurston's compactification, as we can see in
 the following lemma.
\begin{lemma}[Miyachi \cite{Miyachi}] \label{Lemma:Miyachi}
For any $P\in \partial\overline{\mathcal{T}(S)}^{GM}$, there is a non-negative
continuous function $\mathcal{E}_P(\cdot)$ defined on $\mathcal{MF}$, such that
\begin{enumerate}[(i)]
\item \label{m1}  $$\mathcal{E}_P(t\mu)=t\mathcal{E}_P(\mu)$$ for  any $t>0$ and
$\mu\in \mathcal{MF}$, and

\item \label{m2} There is an injective map from  $\partial\overline{\mathcal{T}(S)}^{GM}$ to $P(\mathbb{R}_+^{\mathcal{S}})$: each $P\in\partial\overline{\mathcal{T}(S)}^{GM}$ corresponds to the projective class of the  function $\mathcal{S} \ni \alpha \mapsto \mathcal{E}_P(\alpha)$.

\item \label{m3} The function $\mathcal{E}_P(\cdot)$ is unique up to multiplication by a positive constant in the following
sense: for any sequence $(X_n)$ in $\mathcal{T}(S)$ converging to $P \in \overline{\mathcal{T}(S)}^{GM}$, there exists a subsequence $(X_{n_j})$ such that
$\mathcal{E}_{X_{n_j}}(\cdot)$ converges to a positive multiple of $\mathcal{E}_P(\cdot)$ uniformly on any compact subsets of
$\mathcal{MF}$. In particular,
$$\lim_{n\to \infty}\frac{\mathrm{Ext}_{X_n}(\mu)^{1/2}}{\mathrm{Ext}_{X_n}(\nu)^{1/2}}=\frac{\mathcal{E}_P(\mu)}{\mathcal{E}_P(\nu)}$$
for all $\mu, \nu\in \mathcal{MF}$ with $\mathcal{E}_P(\nu)\neq 0$.
\end{enumerate}
\end{lemma}
Miyachi \cite{Miyachi} also proved that the projective class of the limit function $\mathcal{E}_P(\cdot)$ is independent of the choice of the base-point $X_0$.  In the following, if a  point $P$ is in $\mathcal{T}(S)$ , we will also define $\mathcal{E}_P(\cdot)$
as the function given in $(\ref{Equation:ext})$.

\bigskip
For any $P\in  \overline{\mathcal{T}(S)}^{GM}$, we define  $$\mathcal{Q}(P)=\sup_{\nu\in \mathcal{MF}}\frac{\mathcal{E}_P(\nu)}{\mathrm{Ext}_{X_0}(\nu)^{1/2}}$$ and
$$\mathcal{L}_P(\cdot): \mathcal{MF}\mapsto \mathbb{R}_+: \mu \to \frac{\mathcal{E}_P(\mu)}{\mathcal{Q}(P)}.$$

Recall that in $(\ref{m3})$ of Lemma \ref{Lemma:Miyachi}, the function $\mathcal{E}_P(\cdot)$ is defined up to multiplication by a postive constant.  Multiplying the function $\mathcal{E}_P(\cdot)$ by a positive constant
does not change the value of $\mathcal{L}_P(\cdot)$ and then $\mathcal{L}_P(\cdot)$ is well-defined (independent of
the choice of $\mathcal{E}_P(\cdot)$).
We may consider $\mathcal{L}_P(\mu)$ as a function of the product space $ \overline{\mathcal{T}(S)}^{GM}\times \mathcal{MF}$.
Note that  Walsh \cite{Walsh} used the geodesic currents theory of  Bonahon \cite{Bonahon} to define an analogous function on Thurston's compactification, given by
$$ (P, \mu) \to \frac{I_P(\mu)}{\sup_{\nu\in \mathcal{MF}}\frac{I_P(\nu)}{\ell_{X_0}(\nu)}},$$
where $I_P(\nu)=\ell_P(\nu)$ if $P\in \mathcal{T}(S)$ and $I_P(\nu)=i(P, \nu)$ if $P\in \mathcal{PMF}=\partial \overline{\mathcal{T}(S)}^{Th}$.

The importance of $\mathcal{L}_P(\mu)$ is indicated by the following theorem.

\begin{theorem}\label{Lem:converge}
A sequence $(P_n)$ in $ \overline{\mathcal{T}(S)}^{GM}$ converges to a point $P\in  \overline{\mathcal{T}(S)}^{GM}$
if and only if $\mathcal{L}_{P_n}$ converges to $\mathcal{L}_{P}$ uniformly on compact set of $\mathcal{MF}$.
\end{theorem}
\begin{proof}
($\Rightarrow$) First assume that $ P\in \mathcal{T}(S)$ and $P_n\to P$. By definition,
$$\mathcal{Q}(P)=\sup_{\nu\in \mathcal{MF}}\frac{\mathrm{Ext}_P^{1/2}(\nu)}{K_P^{1/2}\mathrm{Ext}_{X_0}(\nu)^{1/2}}=1.$$
As a result,

\begin{equation}\label{Equation:pp}
\mathcal{L}_P(\mu)=\frac{\mathrm{Ext}_P^{1/2}(\mu)} {K_P^{1/2}} .
\end{equation}
As $n$ is sufficiently large, $P_n\in \mathcal{T}(S)$.
The convergence of  $\mathcal{L}_{P_n}$  to $\mathcal{L}_{P}$ follows directly from
the continuity of the extremal length function.

Now we assume that $P_n\in \mathcal{T}(S), n=1,2, \cdots,  P\in  \partial\overline{\mathcal{T}(S)}^{GM}$ and $P_n$
converges in the Gardiner-Masur compactification to $P$. Let $(P_{n_j})$ be any subsequence of $(P_n)$ such that
for some $t_0>0$, $\mathcal{E}_{P_{n_j}}(\cdot)$ converges to $t_0\mathcal{E}_P(\cdot)$ uniformly on  compact subsets of
$\mathcal{MF}$.

Since
$$\mathcal{Q}(P_n)=\sup_{\nu\in \mathcal{MF}}\frac{\mathcal{E}_{P_n}(\nu)}{\mathrm{Ext}_{X_0}(\nu)^{1/2}}=\sup_{\nu\in \mathcal{PMF}}\frac{\mathcal{E}_{P_n}(\nu)}{\mathrm{Ext}_{X_0}(\nu)^{1/2}},$$

 $\mathcal{Q}(P_{n_j})$ converges to $t_0\mathcal{Q}(P)$. Therefore,
$$\mathcal{L}_{P_{n_j}}(\mu)
=\frac{\mathcal{E}_{P_{n_j}}(\mu)}{\mathcal{Q}(P_{n_j})}$$
converges to $\frac{t_0\mathcal{E}_{P}(\mu)}{t_0\mathcal{Q}(P)}=\mathcal{L}_P(\mu)$ uniformly on compact set of $\mathcal{MF}$. Since the limit is independent of the choice of the subsequence $(P_{n_j})$,
 $\mathcal{L}_{P_{n}}$ converges to $\mathcal{L}_P$ uniformly on compact sets of $\mathcal{MF}$.

For the general case, assume that $P_n  \in \overline{\mathcal{T}(S)}^{GM}$ and $P_n$
converges in the Gardiner-Masur compactification to $P\in  \partial\overline{\mathcal{T}(S)}^{GM}$.
It suffices to show that for any fixed compact set $K$ of $\mathcal{MF}$ and for any $\epsilon>0$ , there exists an $N$, such that for any $n>N$,
$|\mathcal{L}_{P_{n}}(\cdot)-\mathcal{L}_{P}(\cdot)|<\epsilon$ uniformly on $K$.

By the above argument, for each $P_n$, there exists a sequence $(P_{n,k})_{k=1, \cdots, \infty}$ in $\mathcal{T}(S)$ such that $\lim_{k\to\infty} P_{n,k}= P_n$ and $\mathcal{L}_{P_{n,k}}(\cdot)$ converges to $\mathcal{L}_{P_{n}}(\cdot)$ uniformly on compact sets of $\mathcal{MF}$.

Set $P_n'=P_{n,n}$. We claim that $P_n'$ converges to $P$. Otherwise, suppose that there is a subsequence of $P_n'$, still denoted by $P_n'$, which converges to some limit $Q\neq P$.
We have

\begin{enumerate}[(4.1)]
\item $\mathcal{L}_P  \neq \mathcal{L}_Q$ and then there exists a constant $\delta_0$ and a measured foliation $\mu_0$ such that $|\mathcal{L}_P(\mu_0) - \mathcal{L}_Q(\mu_0)| > \delta_0$.
\item $\mathcal{L}_{P_{n}'}(\cdot)$ converges to $\mathcal{L}_Q(\cdot)$ uniformly on compact sets of $\mathcal{MF}$. In particular,
$$|\mathcal{L}_{P_{n}'}(\mu_0) - \mathcal{L}_Q(\mu_0)|< \delta_0/3$$
for $n$ sufficiently large.
\item By definition of $P_n$ and $P_n'$,
$$|\mathcal{L}_{P_{n}}(\mu_0) - \mathcal{L}_{P}(\mu_0)|< \delta_0/3,$$
$$|\mathcal{L}_{P_{n}'}(\mu_0) - \mathcal{L}_{P_{n}}(\mu_0)|< \delta_0/3$$
for $n$ sufficiently large.

\end{enumerate}
From $(4.2), (4.3)$ and the triangle inequality,
$$|\mathcal{L}_P(\mu_0) - \mathcal{L}_Q(\mu_0)| < \delta_0,$$
which contradicts with $(4.1)$.

As a result, $P_n'$ converges to $P$. There is a sufficiently large $N$ such that
for any $n>N$, we have
$$|\mathcal{L}_{P_{n}}(\cdot)-\mathcal{L}_{P_{n}'}(\cdot)|<\frac{\epsilon}{2},  \  |\mathcal{L}_{P_{n}'}(\cdot)-\mathcal{L}_{P}(\cdot)|<\frac{\epsilon}{2}$$
uniformly on $K$.
It follows that
$|\mathcal{L}_{P_{n}}(\cdot)-\mathcal{L}_{P}(\cdot)|<\epsilon$ uniformly on  $K$.

\bigskip
($\Leftarrow$) For $P_n\in \overline{\mathcal{T}(S)}^{GM} (n=1,2, \cdots)$ and $  P$ in
$ \partial\overline{\mathcal{T}(S)}^{GM}$, if $\mathcal{L}_{P_n}$ converges to $\mathcal{L}_{P}$ uniformly on compact set of $\mathcal{MF}$, we want to show that $P_n$ converges to $P$. Let $(Y_n)$ be a subsequence of $(P_n)$
converging in $\overline{\mathcal{T}(S)}^{GM}$ to a point $Y$. From the above discussion, we have
that $\mathcal{L}_{Y_n}$ converges to $\mathcal{L}_Y$ uniformly on any compact set of $\mathcal{MF}$.
Combining this with our assumption that $\mathcal{L}_{P_n}$ converges to $\mathcal{L}_{P}$, we have $\mathcal{L}_Y=\mathcal{L}_P$; that is, for any $\mu\in \mathcal{MF}$,
 $$\frac{\mathcal{E}_Y(\mu)}{\mathcal{Q}(Y)}= \frac{\mathcal{E}_P(\mu)}{\mathcal{Q}(P)}, $$
 or equivalently,
 $$\mathcal{E}_Y(\mu)=\frac{\sup_{\nu\in \mathcal{MF}}\frac{\mathcal{E}_P(\nu)}{\mathrm{Ext}_{X_0}(\nu)^{1/2}}}
 {\sup_{\nu\in \mathcal{MF}}\frac{\mathcal{E}_Y(\nu)}{\mathrm{Ext}_{X_0}(\nu)^{1/2}}} \mathcal{E}_P(\mu).$$
Therefore, $\mathcal{E}_Y$ equals to $\mathcal{E}_P$ up to a positive constant. By $(\mathrm{ii})$ of Lemma \ref{Lemma:Miyachi}, they represent the same point in
$\overline{\mathcal{T}(S)}^{GM}$.
As we have shown that any convergent subsequence of $(P_n)$ converges to $P$, it follows that $P_n$ converges to $P$.
\end{proof}

\begin{lemma}\label{Lem:limit}
Let $(X_n)$ be a sequence of points in $\mathcal{T}(S)$ that converges to a point $P$ in the Gardiner-Masur boundary.
Let $Y$ be a point in $\mathcal{T}(S)$. Let $(\mu_n)$ be a sequence in $\mathcal{PMF}$ such that
$$d_T(X_n,Y)=\frac{1}{2}\log \frac{\mathrm{Ext}_Y(\mu_n)}{\mathrm{Ext}_{X_n}(\mu_n)}.$$ Then any limit point
$\mu_\infty\in \mathcal{PMF}$ of a convergent subsequence of the sequence $(\mu_n)$ satisfies $\mathcal{E}_P(\mu_\infty)=0$.
\end{lemma}
\begin{proof}
 Since $(X_n)$ converges to $P$,  by $(\mathrm{iii})$ of Lemma \ref{Lemma:Miyachi}, there exists a subsequence, still denoted by $(X_{n})$, such that
$\mathcal{E}_{X_{n}}(\cdot)$ converges to  $t_0\mathcal{E}_P(\cdot)$ (for some constant $t_0>0$) uniformly on compact subsets of
$\mathcal{MF}$.

For any limit point $\mu_\infty\in \mathcal{PMF}$ of a convergent subsequence of the sequence $(\mu_n)$, if  $\mathrm{Ext}_P(\mu_\infty)\ne 0$,  then
the function $\mathrm{Ext}_Y(\mu_n)^{1/2}/\mathcal{E}_{X_n}(\mu_n)$ converges to
$\mathrm{Ext}_Y(\mu_\infty)^{1/2}/t_0\mathcal{E}_{P}(\mu_\infty)$.
On the other hand,
\begin{eqnarray*}
\frac{\mathrm{Ext}_Y(\mu_n)^{1/2}}{\mathcal{E}_{X_n}(\mu_n)}&=&\frac{\mathrm{Ext}_Y(\mu_n)^{1/2}}{\mathrm{Ext}_{X_n}(\mu_n)^{1/2}}K_{X_n}^{1/2} \\
&=&e^{{d_T(X_n,Y)}+d_T(X_n,X_0)}
\end{eqnarray*}
which tends to $\infty$ as $X_n$ tends to the boundary.

As a result,
$\mathrm{Ext}_Y(\mu_\infty)^{1/2}/t_0\mathcal{E}_{P}(\mu_\infty)=\infty$ and then $\mathcal{E}_{P}(\mu_\infty)=0$.
\end{proof}

\begin{lemma}[Minsky \cite{Min}]\label{equ:Minsky}
For any $\mu,\nu\in \mathcal{MF}(S)$ and any $X\in \mathcal{T}(S)$, we have
$$i(\mu,\nu)\leq \mathrm{Ext}_{X}(\mu)^{1/2}\mathrm{Ext}_{X}(\nu)^{1/2}.$$
\end{lemma}
\begin{proof}
We sketch a simple proof. Note that the intersection number $i(\cdot,\cdot)$ is continuous on $\mathcal{MF}(S)\times \mathcal{MF}(S) $. By the density of simple close curves in $\mathcal{MF}(S)$, it suffices to prove the lemma for
any $\alpha, \beta\in \mathcal{S}$. Let $q$ be the one-cylinder Strebel differential of height $1$ on $X$ determined
by  $\alpha$.
By a theorem of Jenkins-Strebel \cite{Strebel}, the extremal length $\mathrm{Ext}_{X}(\alpha)$ is realized by:
$$\mathrm{Ext}_{X}(\alpha)=\frac{L^2_\rho(\alpha)}{A(\rho)},$$
where $\rho=|q|^{1/2}|dz|$.

The complement of vertical critical leaves is a cylinder foliated by circles isotopic to $\alpha$.
Note that the circumference and height of the cylinder are equal to $L_\rho(\alpha)$
and $1$ respectively. As a result, $L_\rho(\alpha)=A(\rho)$ and
$\mathrm{Ext}_{X}(\alpha)={L_\rho(\gamma_1)}=A(\rho)$.

Since the length of $\beta$ measured by
the metric $\rho$ is larger than $i(\alpha,\beta)$, we have a lower bound of $\mathrm{Ext}_{X}(\beta)$:
$$\mathrm{Ext}_{X}(\beta)\geq \frac{i(\alpha,\beta)^2}{A(\rho)}= \frac{i(\alpha,\beta)^2}{\mathrm{Ext}_{X}(\alpha)}.$$
As a result, $\mathrm{Ext}_{X}(\alpha)^{1/2}\mathrm{Ext}_{X}(\beta)^{1/2}\geq i(\alpha,\beta)$.
\end{proof}

\begin{lemma}\label{Lem:intersection}Let $(X_n)$ be a sequence of points in $\mathcal{T}(S)$ that converges to a point $P$ in the Gardiner-Masur boundary.
Let $Y$ be a point in $\mathcal{T}(S)$. Let $(\mu_n)$ be a sequence in $\mathcal{PMF}$ such that
$$d_T(X_n,Y)=\frac{1}{2}\log \frac{\mathrm{Ext}_Y(\mu_n)}{\mathrm{Ext}_{X_n}(\mu_n)}.$$
For any $\nu \in \mathcal{MF}$, if $\mathcal{E}_P(\nu)=0$,
then any limit point
$\mu_\infty\in \mathcal{PMF}$ of a convergent subsequence of the sequence $(\mu_n)$ satisfies  $i(\nu,\mu_\infty)=0$.
\end{lemma}
\begin{proof}
Consider any subsequence of $(X_n)$, still denoted by $(X_n)$, such that $\mathcal{E}_{X_{n}}(\cdot)$ converges to  $t_0\mathcal{E}_P(\cdot)$
(for some constant $t_0>0$) uniformly on any compact subset of
$\mathcal{MF}$.

By  Lemma \ref{equ:Minsky}, we have
\begin{eqnarray*}
i(\nu, \mu_n)&\leq& \mathrm{Ext}_{X_n}(\nu)^{1/2}\mathrm{Ext}_{X_n}(\mu_n)^{1/2}.
\end{eqnarray*}
Note that
\begin{eqnarray*}
&& \mathrm{Ext}_{X_n}(\nu)^{1/2}\mathrm{Ext}_{X_n}(\mu_n)^{1/2} \\
&=& K_{X_n}^{1/2}\frac{\mathrm{Ext}_{X_n}(\nu)^{1/2}}{K_{X_n}^{1/2}}\frac{\mathrm{Ext}_{X_n}(\mu_n)^{1/2} }{\mathrm{Ext}_{Y}(\mu_n)^{1/2} }\mathrm{Ext}_{Y}(\mu_n)^{1/2} \\
&=& e^{d_T(X_0,X_n)}\mathcal{E}_{X_n}(\nu) \frac{1}{e^{d_T(Y,X_n)}}\mathrm{Ext}_Y(\mu_n)^{1/2}\\
&\leq&\sup_{\mu\in\mathcal{PMF}}\{\mathrm{Ext}_{Y}(\mu)^{1/2} \} \mathcal{E}_{X_n}(\nu)e^{d_T(X_0,X_n)-d_T(Y,X_n)} \\
&\leq& C e^{d_T(X_0,Y)}\mathcal{E}_{X_n}(\nu),
\end{eqnarray*}
where $C$ is a constant depending on $Y$.
Taking a limit, we have
$$i(\nu, \mu_\infty)\leq Ct_0 e^{d_T(X_0,Y)}\mathcal{E}_{P}(\nu).$$
As a result,  if
 $\mathcal{E}_P(\nu)=0$, then $i(\nu,\mu_\infty)=0$.
\end{proof}

From the above lemmas, we have
\begin{corollary}
Let $(X_n)$ be a sequence of points in $\mathcal{T}(S)$ that converges to a point $P$ in the Gardiner-Masur boundary.
Let $Y$ be a point in $\mathcal{T}(S)$. Let $(\mu_n)$ be a sequence in $\mathcal{PMF}$ such that
$$d_T(X_n,Y)=\frac{1}{2}\log \frac{\mathrm{Ext}_Y(\mu_n)}{\mathrm{Ext}_{X_n}(\mu_n)}.$$
Let $\mu_\infty, \mu_\infty'\in \mathcal{PMF}$ be limit points of two different convergent subsequences of the sequence $(\mu_n)$. Then  $i(\mu_\infty, \mu_\infty')=0$.
\end{corollary}

\section{Proof of Theorem \ref{Thm:main}}\label{main}
Our proof of Theorem \ref{Thm:main} is inspired by Walsh \cite{Walsh}. He showed that
Thuston's compactification $\overline{\mathcal{T}(S)}^{Th}$ of $\mathcal{T}(S)$ is homeomorphic to the horofunction
compactification of $\mathcal{T}(S)$ with Thurston's Lipschitz metric. Moreover,
for each $\mu \in \mathcal{PMF}= \partial\overline{\mathcal{T}(S)}^{Th}$ (Thurston's boundary), the corresponding horofunction, which we denote
by $\Psi_\mu^{Th}$, is given by
\begin{equation}\label{Equation:horo-Th}
\Psi_\mu^{Th}(X)=\log \sup_{\nu\in \mathcal{MF}}\frac{i(\mu,\nu)}{\ell_X(\nu)}-\log \sup_{\nu\in \mathcal{MF}}\frac{i(\mu,\nu)}{\ell_{X_0}(\nu)},
\end{equation}
where $X_0$ is a fixed base point in $\mathcal{T}(S)$.

To prove Theorem \ref{Thm:main}, we will construct the horofunctions of $\mathcal{T}(S)$ with the Teichm\"uller metric by
replacing the  hyperbolic length $\ell_X(\mu)$ with the square root of the extremal length
$\mathrm{Ext}_X(\mu)^{1/2}$ and the intersection number $i(\mu,\cdot)$ by the function $\mathcal{E}_\mu(\cdot)$ (defined in Section 4).

For each $P\in \overline{\mathcal{T}(S)}^{GM}$, we define the map
\begin{eqnarray*}
\Psi_P(X)&=&\log \sup_{\mu\in \mathcal{MF}} \frac{\mathcal{E}_P(\mu)}{\mathrm{Ext}_X(\mu)^{1/2}}-\log \sup_{\mu\in \mathcal{MF}} \frac{\mathcal{E}_P(\mu)}{\mathrm{Ext}_{X_0}(\mu)^{1/2}} \\
&=& \log \sup_{\mu\in \mathcal{MF}} \frac{\mathcal{L}_P(\mu)}{\mathrm{Ext}_X(\mu)^{1/2}}
\end{eqnarray*}
for all $X\in \mathcal{T}(S)$. The last equality follows by the definition
$$\mathcal{L}_P(\mu)=\mathcal{E}_P(\mu) /  \sup_{\mu\in \mathcal{MF}} \frac{\mathcal{E}_P(\mu)}{\mathrm{Ext}_{X_0}(\mu)^{1/2}} .$$

Note that if $P\in \mathcal{T}(S)$, by $(\ref{Equation:ext})$ and Kerckhoff's formula,
\begin{eqnarray*}
\Psi_P(X)&=& \log \sup_{\mu\in \mathcal{MF}} \frac{\mathrm{Ext}_P(\mu)^{1/2}}{K_P^{1/2}\mathrm{Ext}_X(\mu)^{1/2}}-
\log \sup_{\mu\in \mathcal{MF}} \frac{\mathrm{Ext}_P(\mu)^{1/2}}{K_P^{1/2}\mathrm{Ext}_{X_0}(\mu)^{1/2}} \\
&=&  \log \sup_{\mu\in \mathcal{MF}} \frac{\mathrm{Ext}_P(\mu)^{1/2}}{\mathrm{Ext}_X(\mu)^{1/2}}-
\log \sup_{\mu\in \mathcal{MF}} \frac{\mathrm{Ext}_P(\mu)^{1/2}}{\mathrm{Ext}_X(\mu)^{1/2}} \\
&=& d_T(X,P)-d_T(X_0,P).
\end{eqnarray*}
Thus, the function $\Psi_P$ coincides with the function  defined in $(\ref{Equation:Psi})$ for the case where the metric space is $(\mathcal{T}(S),d_T)$.
We will show that it is injective
and continuous. Then Theorem \ref{Thm:main} will follow from a topological argument.

\begin{proposition}\label{pro:inj}
The map $\Psi: \overline{\mathcal{T}(S)}^{GM} \to C(\mathcal{T}(S)): P \to \Psi_P$ is injective.
\end{proposition}

\begin{proof} To prove the assertion, it suffices to prove that for any two distinct points $P, Q\in \overline{\mathcal{T}(S)}^{GM}$,
there exists a point $X\in \mathcal{T}(S)$ such that $\Psi_P(X)\neq \Psi_Q(X)$.

By Theorem \ref{Lem:converge},  $\mathcal{L}_P$ and $\mathcal{L}_Q$ are distinct. Without loss of generality, we assume that
$\mathcal{L}_P(\mu)< \mathcal{L}_Q(\mu)$ for some $\mu\in \mathcal{PMF}$.
Since $\mathcal{L}_P$ and $\mathcal{L}_Q$ are
continuous, we can take a neighborhood $\mathcal{N}$ of $\mu$ in $\mathcal{PMF}$ and real numbers $u$ and $v$ such that
\begin{equation}\label{Equation:PQ}
\mathcal{L}_P(\nu)\leq A < B \leq \mathcal{L}_Q(\nu).
\end{equation}
 for all $\nu\in \mathcal{N}$.

We recall that $\mathcal{PMF}= \partial\overline{\mathcal{T}(S)}^{Th}\subset \partial\overline{\mathcal{T}(S)}^{GM}$.
Since the set of uniquely ergodic measured foliations is dense in $\mathcal{PMF}$, we can choose a  uniquely ergodic measured foliation
$\mu_0 \in \mathcal{N}$. By identifying $\mathcal{PMF}$ with a subset of the Gardiner-Masur boundary, we also choose $\mu_0$ such that $P\neq \mu_0$. By Theorem 3 in \cite{Miyachi2},  this condition of $\mu_0$ means that $\mathcal{L}_P(\mu_0)\neq 0$. Therefore, there is an $M_1>0$ such that $\mathcal{L}_P(\mu_0)\geq M_1 \mathcal{L}_P(\nu)$ for any $\nu\in \mathcal{PMF}$.

Let $r(t)$  be a  Teichm\"uller  geodesic ray defined by a quadratic differential $q$ on $X_0$ with uniquely ergodic vertical measured foliation $\mathcal{F}_v(q)=\mu_0$.
By Theorem \ref{thm:Miyachi}, $r(t)$ converges to $\mu_0$ (considered as a point  in the Gardiner-Masur boundary), with  $\mathcal{E}_{\mu_0}(\cdot)=i(\mu_0, \cdot)$ up to a positive multiplicative constant.

From the sharpness of Minsky's inequality (see Theorem 5.1 in  \cite{GM}), we have
$$\sup_{\nu \in\mathcal{MF}}\frac{i(\mu_0,\nu)}{\mathrm{Ext}_{X_0}^{1/2}(\nu)}=\mathrm{Ext}_{X_0}^{1/2}(\mu_0)=1.$$ It follows that $\mathcal{L}_{r(t)}(\cdot)$ converges to $\mathcal{L}_{\mu_0}(\cdot)$ uniformly on compact sets of $\mathcal{MF}$ and
$$\mathcal{L}_{\mu_0}(\cdot)=i(\mu_0, \cdot).$$

We claim that for  $t$ sufficiently large, the supremum of $\mathcal{L}_{P}(\cdot)/\mathrm{Ext}_{r(t)}^{1/2}(\cdot)$ is attained in the set $\mathcal{N}$.
To see this, note that there are $T_0>0$ and $0<M_2<1$ such that $\mathcal{L}_{r(t)}(\nu)>M_2$ for all $\nu\in \mathcal{PMF}\setminus\mathcal{N}$
but $\mathcal{L}_{r(t)}(\mu_0)< M_1M_2$ for $t\geq T_0$. Therefore,
\begin{eqnarray*}
\frac{\mathcal{L}_{P}(\mu_0)}{\mathrm{Ext}_{r(t)}^{1/2}(\mu_0)}&=&  K_{r(t)}^{1/2} \frac{\mathcal{L}_{P}(\mu_0)}{\mathcal{L}_{r(t)}(\mu_0)}\\
&>&  K_{r(t)}^{1/2} \frac{\mathcal{L}_{P}(\mu_0)}{M_1M_2}\\
&\geq&  K_{r(t)}^{1/2} \frac{\mathcal{L}_{P}(\nu)}{M_2}\\
&\geq&  K_{r(t)}^{1/2} \frac{\mathcal{L}_{P}(\nu)}{\mathcal{L}_{r(t)}(\nu)}\\
&=& \frac{\mathcal{L}_{P}(\nu)}{\mathrm{Ext}_{r(t)}^{1/2}(\nu)}\\
\end{eqnarray*}
for all $\nu\in \mathcal{PMF}\setminus\mathcal{N}$.

As a result,
\begin{eqnarray*}
\sup_{\mu\in\mathcal{PMF}} \frac{\mathcal{L}_{P}(\mu)}{\mathrm{Ext}_X^{1/2}(\mu)}
&=&\sup_{\mu\in\mathcal{N}} \frac{\mathcal{L}_{P}(\mu)}{\mathrm{Ext}_X^{1/2}(\mu)} \\
&\leq&\sup_{\mu\in\mathcal{N}} \frac{A}{\mathrm{Ext}_X^{1/2}(\mu)}\\
&<&\sup_{\mu\in\mathcal{N}} \frac{B}{\mathrm{Ext}_X^{1/2}(\mu)}\\
&\leq& \sup_{\mu\in\mathcal{PMF}} \frac{\mathcal{L}_{Q}(\mu)}{\mathrm{Ext}_X^{1/2}(\mu)}.
\end{eqnarray*}
Thus $\Psi_P(X)< \Psi_Q(X)$.

\end{proof}

The following topological lemma will be used later. The proof here is given by Walsh \cite{Walsh}.
\begin{lemma}\label{lemma:top}
Let $X$ and $Y$ be two topological spaces and let $\Psi : X \times Y \to \mathbb{R} $ be a
continuous function. Let $(x_n)$ be a sequence in $X $converging to $x \in X$. Then,
$\Psi(x_n, \cdot)$ converges to $\Psi(x, \cdot)$ uniformly on compact sets of $Y$.
\end{lemma}
\begin{proof}
Let $K$ be any compact subset of $Y$. For any $\epsilon>0$ and $y\in Y$, there are
open neighborhoods $x\in U_y \subset X, y\in V_y\subset Y$, such that
$$|\Psi(x',y')-\Psi(x,y)|<\epsilon, \ \forall \ x'\in U_y, y'\in V_y.$$
Since $K$ is compact, it can be covered by a finite collection of such neighborhoods $\{V_{y_1}, \cdots, V_{y_n}\}$.
Set $U=\cap_{i=1,\cdots, n} U_{y_i}$. Then for any $x\in U$, $|\Psi(x',y')-\Psi(x,y)|< \epsilon$ holds uniformly for all $y\in K$.
\end{proof}

\begin{lemma}
The map $\Psi: \overline{\mathcal{T}(S)}^{GM} \to C(\mathcal{T}(S)): P \to \Psi_P$ is continuous.
\end{lemma}\label{Lem:continuous}
\begin{proof}
Let $P_n$ be a sequence of $\overline{\mathcal{T}(S)}^{GM} $ converging to a point $P_0$ in $\overline{\mathcal{T}(S)}^{GM}$.
By Theorem \ref{Lem:converge}, $\mathcal{L}_{P_n}$ converges to $\mathcal{L}_{P_0}$ uniformly on compact sets of $\mathcal{MF}$.
For all $X\in \mathcal{T}(S)$, the square root of the extremal function $\mathrm{Ext_X}^{1/2}$ is bounded away from zero on $\mathcal{PMF}$.
We conclude that for any $X\in \mathcal{T}(S)$, $\frac{\mathcal{L}_{P_n}}{\mathrm{Ext}_X^{1/2}}$ converges uniformly on $\mathcal{PMF}$
to $\frac{\mathcal{L}_{P_0}}{\mathrm{Ext}_X^{1/2}}$. Since $$\Psi_{P}(X)=\log \sup_{\mu\in \mathcal{PMF}} \frac{\mathcal{L}_P(\mu)}{\mathrm{Ext}_X(\mu)^{1/2}},$$ $\Psi_{P_n}$ converges pointwise to $\Psi_{P_0}$.
Since the function $\Psi:\overline{\mathcal{T}(S)}^{GM}  \times \mathcal{T}(S)\to \mathbb{R}: (P, X)\to \Psi_P(X)$ is
continuous, and $P_n$ converges to $P_0$ in $\overline{\mathcal{T}(S)}^{GM}$,  by Lemma \ref{lemma:top},   $\Psi_{P_n}(\cdot)$ converges to $\Psi_{P_0}(\cdot)$ uniformly on any compact set of $\mathcal{T}(S)$.
By the definition of the topology of $C(\mathcal{T}(S))$, the map $\Psi: P\to \Psi_P(\cdot)$ is continuous.
\end{proof}

\begin{theorem}
The map $\Psi$ is a homeomorphism between  the horofunction compactification of $\mathcal{T}(S)$ with the Teichm\"uller metric and the Gardiner-Masur compactification of $\mathcal{T}(S)$.
\end{theorem}
\begin{proof}
We have shown that $\Psi: \overline{\mathcal{T}(S)}^{GM}\to C(\mathcal{T}(S)) $ is injective and continuous.
Note that an embedding from a compact space to a Hausdorff space must be a homeomorphism
to its image (see Kelley \cite{Kelley}, Page 141 for the proof). As a result, $\Psi(\overline{\mathcal{T}(S)}^{GM})$ is a compact subset of $C(\mathcal{T}(S))$.
Since the horofuction compactification is the closure of $\Psi(\mathcal{T}(S))$, it is equal to
 $\Psi(\overline{\mathcal{T}(S)}^{GM})$.
\end{proof}

Miyachi \cite{Miyachi}  constructed an embedding $\Phi$ from $\overline{\mathcal{T}(S)}^{GM}$ into
the space $C(\mathcal{MF}(S))_{\geq 0}$ of non-negative continuous functions on $\mathcal{MF}(S)$, defined by
$$\Phi_P(\cdot):= \frac{\mathcal{E}_P(\cdot)}{\mathcal{E}_P(\alpha)+\mathcal{E}_P(\beta)}$$
where $\alpha,\beta$ is a pair of simple closed curves filling $S$. Such an embedding allows us to give a distance
function on $\overline{\mathcal{T}(S)}^{GM}$ by
$$\mathrm{dist}(P, Q)=\max_{\mu\in \mathcal{PMF}}|\Phi_P(\mu)-\Phi_Q(\mu)|.$$
The topology induced by the distance is compatible with the topology given by the embedding $\Psi(\overline{\mathcal{T}(S)}^{GM})\subseteq C(\mathcal{T}(S))$.

\end{document}